\documentclass[english]{article}
\usepackage[T1]{fontenc}
\usepackage[latin9]{inputenc}
\usepackage{babel}
\usepackage{hyperref}

\title{Grid Representations and the Chromatic Number\footnote{This research was supported by the grant SVV-2012-265313 (Discrete Models and Algorithms).} \footnote{A preliminary version appeared in EuroCG 2012 28th European Workshop on Computational Geometry.}}

\author{Martin Balko\thanks{Department of Applie Mathematics, Faculty of Mathematics and Physics, Charles University in Prague. \href{mailto:martin.balko@seznam.cz}{martin.balko@seznam.cz}}}

\usepackage{graphicx}
\usepackage{amsthm}
\usepackage{amsfonts}
\usepackage{amssymb}
\usepackage{complexity}

\theoremstyle{plain}
\newtheorem{thm}{Theorem}
  \theoremstyle{definition}
  \newtheorem{defn}[thm]{Definition}
  \theoremstyle{plain}
  \newtheorem{cor}[thm]{Corollary}
 \theoremstyle{plain}
  \newtheorem{obs}[thm]{Observation}
  \theoremstyle{plain}
  \newtheorem{lem}[thm]{Lemma}
  \theoremstyle{plain}
  \newtheorem{prop}[thm]{Proposition}
  \theoremstyle{plain}
  \newtheorem{conj}[thm]{Conjecture}

\newenvironment{example}[1][Example]{\begin{trivlist}
\item[\hskip \labelsep {\bfseries #1}]}{\end{trivlist}}

\begin{document}
\maketitle

\begin{abstract}
A grid drawing of a graph maps vertices to grid points and edges to line segments that avoid grid points representing other vertices. We show that there is a number of grid points that some line segment of an arbitrary grid drawing must intersect. This number is closely connected to the chromatic number. Second, we study how many columns we need to draw a graph in the grid, introducing some new $\NP$-complete problems. Finally, we show that any planar graph has a planar grid drawing where every line segment contains exactly two grid points. This result proves conjectures asked by David Flores-Pe\~naloza and Francisco Javier Zaragoza Martinez.
\end{abstract}

\section{Introduction}

Let $G=(V,E)$ be a simple, undirected and finite graph. A $k$-{\sl coloring} of $G$ is a function $f\colon V \to C$ for some set $C$ of $k$ colors such that $f \left( u \right) \neq f \left( v \right)$ for every edge $uv\in E$. If such k-coloring of $G$ exists, then $G$ is $k$-{\sl  colorable}. The {\sl  chromatic number} $\chi \left( G \right)$ of $G$ is the least $k$ such that $G$ is $k$-colorable.

For integer $d \geq 2$, a column in the grid $\mathbb{Z}^{d}$ with {\sl  rank} $\left(x_{1},\dots, x_{d-1}\right) \in \mathbb{Z}^{d-1}$ is the set $\left\{ \left( x_{1},\dots, x_{d-1},x\right) \mid x \in \mathbb{Z} \right\}$.  Let $\overline{xy}$ denote the closed line segment joining two grid points $x,y \in \mathbb{Z}^{d}$. The line segment $\overline{xy}$ is {\sl  primitive} if $\overline{xy} \cap \mathbb{Z}^{d}=\left\{ x,y \right\}$ .

\begin{defn}
 A grid drawing $\phi \left( G \right)$ of $G$  in $\mathbb{Z}^{d}$ is an injective mapping $\phi \colon V \to \mathbb{Z}^{d}$ such that, for every edge $uv \in E$ and vertex $w \in V$, $\phi \left( w \right) \in \overline{\phi \left (u \right) \phi \left (v \right)}$ implies that $w = u$ or $w = v$.
\end{defn}

\section{Complexity of the Grid Drawings}

A graph $G$ is said to be {\sl  (grid) locatable in}  $\mathbb{Z}^{d}$ if there exists a grid drawing of $G$ in $\mathbb{Z}^{d}$ where every edge is represented by primitive line segment (such drawing is also called {\sl primitive}). Finding a primitive grid drawing of $G$ is called {\sl locating the graph} $G$. David Flores-Pe\~naloza and Francisco Javier Zaragoza Martinez showed~{\cite{pen09}} the following characterization:

\begin{thm}
[{\cite{pen09}}]
\label{fourcolorableIffLocatable}
A graph $G$ is locatable in $\mathbb{Z}^{2}$ if and only if $G$ is 4-colorable.
\end{thm}

Therefore not all graphs are locatable and every (two-dimensional) grid drawing of any $k$-colorable graph, where $k>4$, contains a line segment which intersects at least three grid points. This led us to a generalization of the concept of locatability. Let the number $gp \left( \phi \left( G \right)  \right)$ denote the maximal number of grid points any line segment of a grid drawing $\phi \left( G \right)$ intersects.

\begin{defn}
A graph $G$ is {\sl (grid)} $q$-locatable in $\mathbb{Z}^{d}$, for some integer $q \geq 2$, if there exists a grid drawing $\phi \left( G \right)$ in $\mathbb{Z}^{d}$ such that $gp \left( \phi \left( G \right)  \right) \leq q$.
\end{defn}

\begin{figure}[h!]
	\centering
	\includegraphics{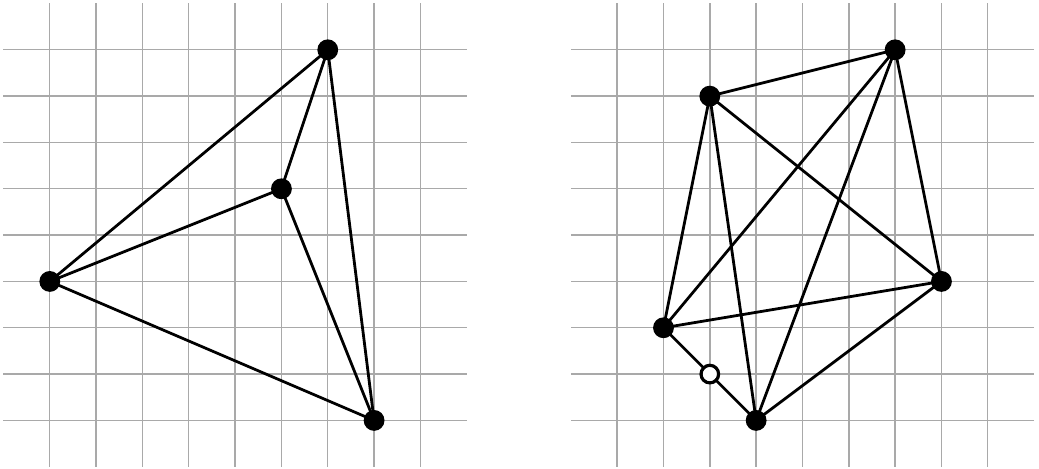}
	\caption{Some grid drawings}
	\label{fig:gridDrawing}
\end{figure}

The complexity of grid drawings of a graph $G$ is understood as the minimum of $gp \left( \phi \left( G \right)  \right)$ among all grid drawings $\phi \left( G \right)$. For example, the graph $K_{5}$ has chromatic number five, thus it is not (two-)locatable. However the grid drawing in Figure~\ref{fig:gridDrawing} shows that $K_{5}$ is three-locatable (the third grid point on line segment is denoted by an empty circle). The main result of this section is a stronger version of Theorem~\ref{fourcolorableIffLocatable}.

\begin{thm}
\label{complexityTheorem}
For integers $d,q \geq 2$, a graph $G$ is $q^{d}$-colorable if and only if $G$ is $q$-locatable in $\mathbb{Z}^{d} $.
\end{thm}

We split the proof of this theorem into two parts. First, we show the easier implication and then, after some auxiliary constructions, we give a proof of the reverse implication.

\begin{lem}
\label{firstImp}
For integers $d,q \geq 2$, if the graph $G$ is $q$-locatable in $\mathbb{Z}^{d}$, then it is $q^{d}$-colorable.
\end{lem}
\begin{proof}
A trivial but useful observation (see {\cite{apos76}} for example) is that the line segment $\overline{ab}$ between the grid points $a,b \in \mathbb{Z}^{d}$ intersects exactly the grid points of the form \[ \left( a_{1}+ i \frac{b_{1}-a_{1}}{\alpha},\dots, a_{d}+ i \frac{b_{d}-a_{d}}{\alpha}\right)\] where $0 \leq i \leq \alpha$ and $\alpha = \gcd \left( |a_{1}-b_{1}|,\ldots,|a_{d}-b_{d}| \right)$. Let $\phi \left( G \right)$ be a grid drawing of the graph $G= \left( V,E \right)$ in $\mathbb{Z}^{d}$ having $gp \left( \phi \left( G \right) \right) \leq q$. Consider the function $f \colon \mathbb{Z}^{d} \to \mathbb{Z}^{d}_{q}$ denoted as 
\[f \left( x_{1},\dots,x_{d} \right) = \left( x_{1} \left( \bmod q \right),\dots,x_{d} \left( \bmod q \right) \right)\] We use $f$ as coloring of the grid with $q^{d}$ colors and we show that it is also a proper vertex coloring of $G$. Assume to the contrary that $f \left( \phi \left( u \right) \right) = f \left( \phi \left( v \right) \right)$ for some $uv \in E$. Then $u_{1} \equiv v_{1}, \dots, u_{d} \equiv v_{d} \left( \bmod q \right)$ which implies \[\gcd \left( |u_{1} - v_{1}|,\dots,|u_{d}-v_{d}| \right) \geq q\] According to our observation, there are at least $q+1$ grid points lying on the line segment $\overline{\phi \left( u \right) \phi \left( v \right)}$. This contradicts the fact that $G$ is $q$-locatable via the drawing $\phi \left( G \right)$. 
\end{proof}

Thus it remains to show the implication in the opposite direction. The main idea is to find a subset of $\mathbb{Z}^{d}$ which we can use for a convenient grid drawing of every $q^d$-colorable graph. 

Assume that the dimension $d$ is fixed and let $p$ be a prime number. We define $V_{p,1}$ as the sequence $\left\{ x_{i} \right\} _{i=0}^{p^{d}-1}$ such that each $x_{i}$ is from the set $\mathbb{Z}^{d}_{p}$ and no two terms are equal. This definition is correct as we can always find $p$ distinct residues modulo $p$ and, naturally, there are $p^{d}$ distinct $d$-tuples of these residues. Now we define $V_{p,e}$ for $e \geq 2$ inductively. Assume as induction hypothesis that we have already set $V_{p,e-1}$. Now we place $V_{p,e}$ as a chain of $p^{d}$ copies of $V_{p,e-1}$. Then we change the terms on the positions \[i+p^{d\left( e -1 \right)},\ldots,i+ \left( p^{d} -1 \right)p^{d\left( e-1 \right)}\] for every $i \in \left\{ 0,1,\ldots,p^{d\left( e-1 \right)}-1 \right\}$ in such way that the new terms are numbers from $\mathbb{Z}_{p^{e}}$ congruent to their predecessors modulo $p^{e}$  and no two terms in $V_{p,e}$ are equal. For each element of $Z^{d}_{p^{e-1}}$ there are $p^{d}$ congruent elements from $\mathbb{Z}_{p^{e}}^{d}$ modulo $p^{e-1}$ and one of them is on the $i$-th position of $V_{p,e}$. Thus the definition of $V_{p,e}$ is, again, correct.

Continual repeating of the copies of $V_{p,e}$ gives us the infinite sequence $S_{p,e}$. We denote the $i$-th term of $S_{p,e}$ as $S_{p,e}[i]$ and the distance of two terms $S_{p,e}[i]$ and $S_{p,e}[j]$ is given by $|i-j|$. The following lemma shows an important feature of these sequences.
\begin{lem}
\label{sequenceDist}
Let $p$ be prime number and $e$ positive integer. Then two terms of $S_{p,e}$ are equal if and only if $p^{de}$ divides their distance.
\end{lem}
\begin{proof}
Suppose that our terms are on positions $i$ and $j$. The case $i=j$ is apparent, thus we can assume $i \neq j$. From the definition two distinct terms equal if and only if both are in different copies of $V_{p,e}$, but on the same position in $V_{p,e}$. The length of $V_{p,e}$ is exactly $p^{de}$, so the distance between $S_{p,e}[i]$ and $S_{p,e}[j]$ is a multiple of $p^{de}$.
\end{proof}

Given a number $s$, we set $f \left( p \right)$ as $\min \left\{ e \in \mathbb{N} \mid p^{de}\geq s\right\}$ for every prime number $p<s$. Now, for every $i$, where $0 \leq i \leq s-1$, we choose a distinct column of $\mathbb{Z}^{d}$ such that for every prime number $p<s$ the rank of this column is congruent to the first $d-1$ elements of the $d$-tuple $S_{p,e}[i]$ modulo $p^{f \left( p \right)}$. We label the chosen columns as $W_{0,s},W_{1,s},\ldots,W_{s-1,s}$. In every column $W_{i,s}$ we keep only the points with their last coordinate congruent to the last element of $S_{p,e}[i]$ modulo $p^{f \left( p \right)}$, again for every $p<s$. Finally we set $W_{s}=\bigcup_{i=0}^{s-1} W_{i,s}$.

Let us mention the last technical remark. If there is a prime $p \geq s$ such that ranks of two or more columns from $W_{s}$ are congruent modulo $p$, then we assign distinct residues modulo $p$ to these columns. Subsequently, we keep only the points with their last coordinate congruent to the assigned residue modulo $p$ in each one of these columns. This method is correct, because the number of possible residues is at least $s$, thus every column can get an unique residue. According to the Chinese Remainder Theorem, every column of $W_{s}$ still contains infinitely many points.

\begin{example}

Assume we want to build $W_{9}$ in two-dimensional case. For $s=9$, we have to define the sequences $S_{2,2}$, $S_{3,1}$, $S_{5,1}$ and $S_{7,1}$, as $2^{4},3^{2},5^{2},7^{2} \geq 9$. No other sequences are required, because $9 \leq p$ for every other prime number $p$.

\begin{figure}[h!]
\centering
$S_{2,2} = \left( 0,0 \right), \left( 0,1 \right), \left( 1,0 \right), \left( 1,1 \right), \left( 0,2 \right), \left( 0,3 \right), \left( 1,2 \right), \left( 1,3 \right), \left( 2,0 \right), \ldots$

$S_{3,1} = \left( 0,0 \right), \left( 0,1 \right), \left( 0,2 \right), \left( 1,0 \right), \left( 1,1 \right), \left( 1,2 \right), \left( 2,0 \right), \left( 2,1 \right), \left( 2,2 \right), \ldots$

$S_{5,1} = \left( 0,0 \right), \left( 0,1 \right), \left( 0,2 \right), \left( 0,3 \right), \left( 0,4 \right), \left( 1,0 \right), \left( 1,1 \right), \left( 1,2 \right), \left( 1,3 \right), \ldots$

$S_{7,1} = \left( 0,0 \right), \left( 0,1 \right), \left( 0,2 \right), \left( 0,3 \right), \left( 0,4 \right), \left( 0,5 \right), \left( 0,6 \right), \left( 1,0 \right), \left( 1,1 \right), \ldots$
\end{figure}

Then we can get the set $W_{9}$ as a union of the following columns:

\begin{figure}[h!]
\centering
$W_{0,9} = \left\{ \left( 0,x \right) \in \mathbb{Z}^{2} \mid x \equiv 0 \left( \bmod 4 \right), 0 \left( \bmod 3 \right), 0 \left( \bmod 5 \right), 0 \left( \bmod 7 \right) \right\}$

$W_{1,9} = \left\{ \left( 420,x \right) \in \mathbb{Z}^{2} \mid x \equiv 1 \left( \bmod 4 \right), 1 \left( \bmod 3 \right), 1 \left( \bmod 5 \right), 1 \left( \bmod 7 \right) \right\}$

$W_{2,9} = \left\{ \left( 105,x \right) \in \mathbb{Z}^{2} \mid x \equiv 0 \left( \bmod 4 \right), 2 \left( \bmod 3 \right), 2 \left( \bmod 5 \right), 2 \left( \bmod 7 \right) \right\}$

$W_{3,9} = \left\{ \left( 385,x \right) \in \mathbb{Z}^{2} \mid x \equiv 1 \left( \bmod 4 \right), 0 \left( \bmod 3 \right), 3 \left( \bmod 5 \right), 3 \left( \bmod 7 \right) \right\}$

$W_{4,9} = \left\{ \left( 280,x \right) \in \mathbb{Z}^{2} \mid x \equiv 2 \left( \bmod 4 \right), 1 \left( \bmod 3 \right), 4 \left( \bmod 5 \right), 4 \left( \bmod 7 \right) \right\}$

$W_{5,9} = \left\{ \left( 196,x \right) \in \mathbb{Z}^{2} \mid x \equiv 3 \left( \bmod 4 \right), 2 \left( \bmod 3 \right), 0 \left( \bmod 5 \right), 5 \left( \bmod 7 \right) \right\}$

$W_{6,9} = \left\{ \left( 161,x \right) \in \mathbb{Z}^{2} \mid x \equiv 2 \left( \bmod 4 \right), 0 \left( \bmod 3 \right), 1 \left( \bmod 5 \right), 6 \left( \bmod 7 \right) \right\}$

$W_{7,9} = \left\{ \left( 281,x \right) \in \mathbb{Z}^{2} \mid x \equiv 3 \left( \bmod 4 \right), 1 \left( \bmod 3 \right), 2 \left( \bmod 5 \right), 0 \left( \bmod 7 \right) \right\}$

$W_{8,9} = \left\{ \left( 386,x \right) \in \mathbb{Z}^{2} \mid x \equiv 0 \left( \bmod 4 \right), 2 \left( \bmod 3 \right), 3 \left( \bmod 5 \right), 1 \left( \bmod 7 \right) \right\}$
\end{figure}

In the last step we ensure possible occurrences of prime numbers $p \geq s$ in decompositions of differences of ranks. For example, prime number $23$ divides the difference of ranks $0$ and $161$. But, if we can keep only the points $\left( 0,x \right) \in W_{0,9}$ and the points $\left( 0,y \right) \in W_{6,9}$ such that $x$ and $y$ are not congruent modulo $23$, then $23$ does not divide $\gcd \left( |a_{1}-b_{1}|,|a_{2}-b_{2}| \right)$ for any $a \in W_{0,9}$ and $b \in W_{6,9}$.
\label{exampleWs}
\end{example}

The construction of the set $W_{s}$ is not easy to describe, but it has nice properties that allow us to prove the crucial lemma in the proof of Theorem~\ref{complexityTheorem}.
\begin{lem}
\label{setWs}
Let $s$ be $d$-th power of integer $q \geq 2$. Let $a= \left( a_{1},\ldots,a_{d}\right)$, $b= \left( b_{1},\ldots,b_{d}\right)$ be grid points located in distinct columns of the set $W_{s}$. Then \[\gcd \left( \left| a_{1} - b_{1} \right|, \ldots, \left| a_{d} - b_{d}\right| \right) \leq \sqrt[d]{s} - 1\]
\end{lem}
\begin{proof}
Let $\alpha$ denote the greatest common divisor in the statement. Assume that the grid point $a$ is in the column $W_{x,s}$ and the grid point $b$ in the column $W_{y,s}$, $0 \leq x,y \leq s-1$ and $x \neq y$. The last remark in the construction of $W_{s}$ guarantee that no prime number larger than $s-1$ divides $\alpha$. Also, for every $e \in \mathbb{N}$ and prime number $p$, the power $p^{e}$ divides $\alpha$ if and only if $S_{p,e}[x] = S_{p,e}[y]$. Because $p^{e} \mid \alpha$ implies that each coordinate of $a$ is congruent to each coordinate of $b$ modulo $p^{e}$ and these coordinates are congruent to the $d$-tuples $S_{p,e}[x]$ and $S_{p,e}[y]$ modulo $p^{e}$. Thus $p^{e}$ does not divide  $\alpha$ for $e \geq f \left( p \right) $. Otherwise $S_{p,f \left( p \right)}[x] = S_{p,f \left( p \right)}[y]$ and, according to Lemma~\ref{sequenceDist}, the distance $\left| x-y \right|$ between them is at least $p^{d f \left( p \right)}$, which is at least $s$. But this contradicts the inequality $0 \leq x,y \leq s-1$. 

So we can assume that $\alpha = \prod_{i=1}^{k} {p_{i}^{e_{i}}}$ where $p_{i}$ are prime numbers and $1 \leq e_{i} \leq f \left( p_{i} \right) -1$. Then \[\alpha^{d} = \prod_{i=1}^{k} {p_{i}^{de_{i}}} \leq s - 1\] holds. Because the expression of $\alpha$ implies that $S_{p_{i},e_{i}}[x]=S_{p_{i},e_{i}}[y]$ and, again, we get $p_{i}^{de_{i}} \mid \left| x-y \right|$ for every $i \in \left\{ 1,2,\ldots,k \right\}$. Thus $\alpha^{d} \leq \left| x-y \right| \leq s - 1$.

We know that $s$ is $d$-th power of some integer $q \geq 2$ and we just showed that $\alpha^{d}$ is smaller $d$-th power than $s$. Thus $\alpha^{d} \leq {\left( q - 1 \right)}^{d}$. But this gives us the required inequality, as $ \left( q - 1 \right) \leq \sqrt[d]{q^{d}} -1$ holds trivially.
\end{proof}

Now we can finally prove Theorem~\ref{complexityTheorem}.
\begin{proof}[Proof of Theorem~\ref{complexityTheorem}]
The first implication is proven in Lemma~\ref{firstImp}, so assume that $G$ is a $q^{d}$-colorable graph, $q \geq 2$. We need to find a grid drawing of $G$ such that at most $q$ grid points lie on any of its line segments. It suffices to show how to find such drawing for complete $q^{d}$-partite graph $K_{n,\ldots,n}$ and arbitrary $n \in \mathbb{N}$, because every $q^{d}$-colorable graph on $n$ vertices is its subgraph. We consider the set $W_{s}$ for $s=q^{d}$ and we keep only the first $n$ vertices of its first two columns. Then for every $i$, $2\leq i \leq s-1$, we keep the first $n$ points in the column $W_{i,s}$ such that all points in previous columns are visible from any of these points (with respect to other columns). These points in $W_{i,s}$ exist, because, unlike $W_{i,s}$, the previous columns are finite sets. 

Afterwards we obtain the set $W_{s} \left( n \right) \subset W_{s}$ such that $K_{n,\ldots,n}$ is isomorphic to the visibility graph $\upsilon \left(  W_{s} \left( n \right) \right)$ and, according to Lemma~\ref{setWs}, no line segment contains more then $q$ grid points.  Therefore we get suitable grid drawing of $G$ and the second implication is proven.
\end{proof}

Note that the proof is constructive and we can find an appropriate grid drawing in time $O \left( \left| V \right| \right)$ for a given coloring of $G$.

\begin{cor}
\label{locatabilityDim}
A graph is $2^{d}$-colorable if and only if is locatable in $\mathbb{Z}^{d}$, for $d \geq 2$.
\end{cor}

\begin{cor}
For given $d,q \geq 2$, it is \NP-complete to decide whether or not a graph $G$ is $q$-locatable in $\mathbb{Z}^{d}$.
\end{cor}
\begin{proof}
Clearly, the problem belongs to \NP. Theorem~\ref{complexityTheorem} shows a reduction of the colorability problem, which asks ``Does $G$ admit a proper vertex coloring with $q^{d}$ colors?'', to our problem. We can also ensure that the reduction is polynomial.
\end{proof}

\section{Compactness}

Our main concern in this section is how to draw a graph on the bounded number of columns in a grid. There is no loss of generality in assuming that the grid is two-dimensional. Because if we can find a grid drawing $\phi \left( G \right)$ in $\mathbb{Z}^{d}$, $d > 2$, on $l$ columns, then we can transfer this drawing on $l$ columns in $\mathbb{Z}^{2}$. We just take each column of the original grid drawing and transfer its points to an arbitrary free column in the plane. Then we might have to shift some columns higher so that no point representing vertex lies on nonadjacent line segment. This is always possible as the number of vertices in $G$ is finite. By the same trick, we can also assume that there is no unused column between two columns in our drawing. If $l$ is the minimal number of columns on which $G$ can be drawn, then we say that this grid drawing of $G$ is {\sl compact}.

It is easy to see that if there is a grid drawing on $l \geq 2$ columns for a graph $G$, then $G$ is $l$-locatable (in the plane), because the differences of column ranks from such grid drawing are always lower then $l$ and we can move the adjacent points of the same column such that the line segment between them is primitive. The implication in the reverse direction does not hold as the graph $K_{7}$ is, according to Theorem~\ref{complexityTheorem}, three-locatable, but it cannot be drawn on three columns, because the last vertex with any other two vertices induces $C_{3}$. Thus compactness is not the locatability in disguise. Suppose  $G$ is $l$-locatable, then we know it is $l^{2}$-colorable. In such case $G$ is embeddable on $l^{2}$ columns, because the vertices of each color can use one column. Thus we have:
\begin{cor}
\label{columnsChromatic}
A graph $G$ is embeddable on at most $\chi \left( G \right)$ columns.
\end{cor}

However none of the shown bounds is tight, because there is, for example, a locatable graph with a compact grid drawing on three columns. See Figure~\ref{fig:locatable3Columns}. 
\begin{figure}[h!]
	\centering
	\includegraphics{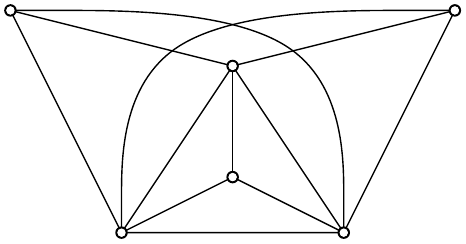}
	\caption{A locatable graph with a compact grid drawing on three columns}
	\label{fig:locatable3Columns}
\end{figure}

The next simple observation characterizes which graphs are embeddable on $l$ columns in the terms of the graph theory.
\begin{obs}
\label{columnsCharacterization}
A graph $G = \left( V,E \right)$ is embeddable on $l$ columns if and only if $V$ can be partitioned into $V_{1},V_{2},\ldots,V_{l}$ such that each induced subgraph $G\left[V_{i}\right]$ is isomorphic to a disjoint union of paths.
\end{obs}
\begin{proof}
In the grid drawing of $G$ on $l$ columns, the vertices represented by points of a single column define a set $V_{i}$. On the other hand, if we have a partition $V_{1},V_{2},\ldots,V_{l}$ of $V$, then each $G\left[V_{i}\right]$ can be drawn on a single column and we can always shift the vertices in such way that the visibility of representing points is guaranteed.
\end{proof}

Thus embedding of a graph on few columns is equivalent with a special variant of defective coloring. That is, an improper vertex coloring in which every color class induces a cycle-free subgraph of maximum degree at most two  (that is a linear forest). We call these color classes {\sl path-colors} for short. Also note that the case $l=1$ is not difficult, because a graph that is embeddable on a single column is a disjoint union of paths and this can be determined in linear time.

If we restrict our attention to only primitive grid drawings, then the situation changes rapidly. According to Theorem~\ref{complexityTheorem}, only four-colorable graphs have primitive grid drawings in the plane. Also we know, according to Corollary~\ref{locatabilityDim}, that we have to proceed to grid drawings in higher dimensions if we want to obtain primitive grid drawings of graphs with larger chromatic number. The minimum dimension of grid on which a graph $G$ can be located is $\lceil \log _{2} \left( \chi \left( G \right) \right) \rceil$ and this is the dimension we factor in for $G$. Despite the fact that the situation with primitive grid drawings is quite different, Theorem~\ref{complexityTheorem} gives us the same upper bound on the minimal number of columns. 

\begin{cor}
\label{locatingChromatic}
A graph $G$ can be located on at most $\chi \left( G \right)$ columns in $\mathbb{Z}^{d}$.
\end{cor}

However this bound is not tight even in the current case.  For example, the graph $K_{5}$ cannot be located in $\mathbb{Z}^{2}$ as its chromatic number is five, but it can be located on three columns in $\mathbb{Z}^{3}$. Note that this number of columns is minimum, because three vertices on a single column induce a 3-cycle. Thus compact primitive grid drawing of $K_{5}$ is on three columns in $\mathbb{Z}^{3}$.

In the previous section we assume that the set of columns in a compact grid drawing does not contain any holes. That is, there are no unused columns between two columns of this grid drawing. But now we cannot modify a primitive grid drawing by the same trick as before, because shifted line segments could intersect more grid points and the drawing would not be primitive then. Thus it could happen that some primitive grid drawings on minimal number of columns are necessarily vast and sparse. Luckily, the following theorem shows that there are primitive grid drawings with minimal number of columns which take up little space. It also gives us a characterization of locating similar to Observation~\ref{columnsCharacterization}.

\begin{thm}
\label{locatingCharacterization}
For a graph $G=\left( V,E \right)$, integers $d \geq 2$ and $l$, $2^{d-1} < l \leq 2^{d}$, the following statements are equivalent:
\begin{enumerate}
\item $G$ can be located on $l$ columns in $\mathbb{Z}^{d}$,
\item $V$ can be partitioned into $V_{1},V_{2},\ldots,V_{l}$ such that $2^{d}-l$ induced subgraphs $G\left[V_{i}\right]$ induce a disjoint union of paths and the rest induces independent sets.
\end{enumerate}
\end{thm}

Note that the dimension of a grid is minimum for such choice of $l$, according to Corollary~\ref{locatabilityDim}. Also an independent set is a disjoint union of paths as well, thus the statement says that there are at most $2^{d}-l$ induced subgraphs $G\left[V_{i}\right]$ that induce a disjoint union of paths.

\begin{proof}
Suppose that $G$ is located on $l$ columns in $\mathbb{Z}^{d}$. We construct a {\sl congruence graph} $C$ on the set of column ranks of such primitive grid drawing. Every vertex of this graph corresponds to an unique column rank and two vertices are adjacent if the corresponding ranks are congruent modulo two. The graph $C$ is a disjoint union of complete graphs, because congruence is equivalence relation. All points in the columns with ranks which lie in the same connected component of $C$ can be colored with two colors and each such color  induces an independent set. Because if we color the points with the odd last coordinate white and the points with the even last coordinate black, then no two monochromatic points can share an edge. Since such ranks are congruent modulo two, then the line segment joining two adjacent monochromatic points would not be primitive. But this would be a contradiction, since the whole grid drawing is primitive. Thus we can use two colors in each clique in $C$ which contains at least two vertices.

Now we show by induction on $l$ that $l$ colors is sufficient and that there are at most $2^{d}-l$ colors that induce a disjoint union of paths. Consider the case when $l=2^{d-1}+1$. Then the congruence graph cannot contain more than $2^{d-1}-1$ isolated vertices, because the maximal number of possible values of ranks modulo two is $2^{d-1}<l$. In such case we color the points of column, whose rank corresponds to an isolated vertex in $C$, with a single color. These colors induce disjoint unions of paths. Then we color the points in all columns with ranks congruent modulo two with only two colors (as we showed before). Then the condition holds, because $2^{d}-l=2^{d}-\left(2^{d-1}+1 \right)=2^{d-1}-1=l-2$. 

 Let us assume that this initial graph contains all isolated vertices of the final congruence graph $C$. Now suppose that our $C$ contains $l$ vertices and we know from the induction hypothesis that the condition holds for congruence graphs on $l-1$ vertices. We get the graph $C$ by joining one vertex $u$ to such congruence graph. Due to the choice of initial graph, we know that $u$ is not isolated in $C$.  If we join $u$ to some clique with at least two vertices, then we color the points of a corresponding column with the two colors of this clique. One color for points in even height, the other one for points in odd length. If $2^{d}-l$ drops bellow the number of colors which induce a disjoint union of paths, then we choose an isolated vertex whose column is monochromatic and color its points with two colors. One color is the original one, the other is new for $u$. If we join the new vertex $u$ to an isolated vertex $v$, then there are two possibilities. If the points of the column with rank $v$ are colored with a single color, then we color points in the columns with ranks $u$ and $v$ using two colors. One is new for $u$, the other is original. If points of the column with rank $v$ are bi-chromatic, then we color points in both columns with these two colors and alternatively correct the case of low $2^{d}-l$ as before.

Now we prove the reverse implication. Let $V_{1},V_{2},\ldots,V_{l}$ be the partition of $V$ in the second statement. Consider the set \[\left\{ \left(r_{1},r_{2},\ldots,r_{d-1}\right)\in\mathbb{Z}^{d-1}\mid r_{1} \in \mathbb{Z}_{4}, r_{i} \in \mathbb{Z}_{2}\right\}\] The last $d-2$ coordinates $r_{2},r_{3},\ldots,r_{d-1}$ determine the set \[\left\{ \left(r_{1},r_{2},\ldots,r_{d-1}\right)\in\mathbb{Z}^{d-1}\mid r_{1} \in \mathbb{Z}_{4}\right\}\] We mark it as $G_{r_{2},\ldots,r_{d-1}}$ and its elements as $g_{i,r_{2},\ldots,r_{d-1}} = \left( i,r_{2},\ldots,r_{d-1}\right)$, for $i=0,1,2,3$. For $d=2$, there is only one such $G=\left\{ 0,1,2,3 \right\}$.  Now we show a simple algorithm how to locate $G$ on columns with ranks from this set. We repeat the following steps until there is no set of vertices left in our partition. 

\begin{enumerate}
\item Take $G_{r_{2},\ldots,r_{d-1}}$ that has not been chosen yet. 
\item If there are two sets $V_{i}$, $V_{j}$ such that $G\left[V_{i}\right]$, $G\left[V_{j}\right]$  are linear forests and there is no set which induces an independent set, then map the vertices from $V_{i}$ to points of column with rank $g_{0,r_{2},\ldots,r_{d-1}}$ and the vertices from $V_{j}$ to points of column with rank $g_{1,r_{2},\ldots,r_{d-1}}$.
\item If there is $V_{i}$ such that $G\left[V_{i}\right]$ induce a linear forest and two sets $V_{j}$, $V_{k}$ which induce independent sets, then map the vertices of $V_{i}$ on the column with rank $g_{1,r_{2},\ldots,r_{d-1}}$. Also, map the vertices of $V_{j}$ to points of column with rank $g_{0,r_{2},\ldots,r_{d-1}}$ that have even $d$-th coordinate and map the vertices of $V_{k}$ to points of column with rank $g_{2,r_{2},\ldots,r_{d-1}}$ that have odd $d$-th coordinate.
\item If there is no such $V_{i}$, then take four (or two, if there are not that many) sets from the partition. Let these sets be $V_{i}$, $V_{j}$, $V_{k}$ and $V_{m}$. Every one of them induces an independent set. Map $V_{i}$ to points of column with rank $g_{0,r_{2},\ldots,r_{d-1}}$ that have even $d$-th coordinate divisible by three and $V_{j}$ to points of column with rank $g_{1,r_{2},\ldots,r_{d-1}}$ that have even $d$-th coordinate too. Then, map $V_{k}$ to points of column with rank $g_{2,r_{2},\ldots,r_{d-1}}$ that have odd last coordinate and $V_{m}$ to points of column with rank $g_{3,r_{2},\ldots,r_{d-1}}$ that have odd last coordinate which is not divisible by three.
\item Remove chosen sets of vertices from the partition.
\end{enumerate}

Note that the total number of sets in the partition which induce independent set is even, because this number equals $l-\left(2^{d}-l\right)=2l - 2^{d}$. Thus if there is at least one such set in any step of the algorithm, then there is also another one, because we remove these sets by two or four.

The maximum number of steps is $2^{d-2}$, because it is also the number of sets $G_{r_{2},\ldots,r_{d-1}}$. We show that this number is sufficient. First, notice that for each $V_{i}$, that induces a linear forest, we lower $l$ by one (if we start with empty partition and $l=2^{d}$). Thus we can pair such $V_{i}$ with unique empty set of vertices and we obtain $2^{d}$ sets of vertices such that some of them induce a disjoint union of paths, some an independent set and some are empty. Each step of the algorithm takes four of these sets and locates their vertices. Thus we can locate all these $2^{d}=4\cdot2^{d-2}$ sets within $2^{d-2}$ steps. 

It is not difficult to see that the obtained grid drawing is primitive as the only possible occurrence of non-primitive line segment is between columns from the same set $G_{r_{2},\ldots,r_{d-1}}$. But we mapped the vertices such that no line can intersect more than two grid points.
\end{proof}

The proof of the previous theorem shows how to relocate a primitive grid drawing of $G$ on minimal number of columns, such that the new grid drawing is still primitive and it also requires small part of the grid (the first $d-1$ coordinates are constant). We also obtained relation between compact and primitive compact grid drawings.

\begin{cor}
\label{compactLocatingRelation}
Every graph with a grid drawing on $l$ columns has a primitive grid drawing on $k$ columns in $\mathbb{Z}^{d}$ where $l \leq k \leq 2l-2^{d-1}$ and $d$ is an integer such that $2^{d-1} +1 \leq l  \leq 2^{d}$.
\end{cor}
\begin{proof}
The lower bound on $k$ is immediate. To show the upper bound we just combine Observation~\ref{columnsCharacterization} and Theorem~\ref{locatingCharacterization}. It suffices to split each of $l-2^{d-1}$ path-colors into two normal colors. Then the final number of columns is
\[l-(l-2^{d-1})+2(l-2^{d-1})=2l-2^{d-1}.\]
\end{proof}

We can also characterize graphs which can be located on less than $2^{d-1}+1$ columns.
\begin{obs}
\label{locatingEquivalentCompact}
For a graph $G=\left( V,E \right)$ and integers $d \geq 2$ and $l$, $1\leq l \leq 2^{d-1}$, the following statements are equivalent:
\begin{enumerate}
\item $G$ can be located on $l$ columns in $\mathbb{Z}^{d}$,
\item $G$ is embeddable on $l$ columns (in $\mathbb{Z}^{2}$).
\end{enumerate}
\end{obs}
\begin{proof}
Let $G$ is located on $2^{d-1}$ columns in $\mathbb{Z}^{d}$. Then we can take each column of this primitive grid drawing and arrange them in a consecutive order in the plane. Then we might have to shift some columns higher to satisfy the condition on mutual visibility with respect to points representing vertices. On the other hand, if there is a grid drawing of $G$ on $2^{d-1}$ columns in the plane, then we take each column of this drawing and copy it on an unique point from the set \[\left\{ \left( r_{1},\ldots,r_{d-1} \right) \mid 0 \leq x_{i} \leq 1 \right\} \subset \mathbb{Z}^{d-1}\]
\end{proof}

This observation is somehow intuitive as every grid drawing on two columns is primitive. However, we know, according to Theorem~\ref{locatingCharacterization}, that for a larger number of columns this does not hold and locating becomes more restrictive than drawing.

Although we show that locating the graph on bounded number of columns is $\NP$-complete in the following section, there are special classes of graphs for which we can find suitable estimations. The following theorem gives bounds that depend on the maximum degree of a graph. In order to show this, we need an auxiliary lemma proven by L\'{a}szl\'{o} Lov\'{a}sz.

\begin{lem}
[{\cite{lov66}}]
\label{lovaszLemma}
Let $G=\left( V,E \right)$ be a graph and let $k_{1},k_{2},\ldots,k_{m}$ be nonnegative integers with $k_{1} + k_{2} + \ldots + k_{m} \geq \Delta \left( G \right) - m + 1$. Then $V$ can be partitioned into $V_{1}, V_{2}, \ldots, V_{m}$ so that $\Delta\left( G \left[ V_{i} \right] \right) \leq k_{i}$, for all $i \in \left[ m \right]$.
\end{lem}

\begin{thm}
\label{locatingDegree}
Let $G = \left( V,E \right)$ be a graph with $\Delta \left( G \right) \leq 2^{d+1} - 1$, for $d \in \mathbb{N}$. Then $G$ can be located on $2^{d}$ columns in $\mathbb{Z}^{d+1}$.
\end{thm}
\begin{proof}
According to Proposition~\ref{locatingEquivalentCompact}, it suffices to prove that $G$  is embeddable on $2^{d}$ columns in the plane. To prove this we apply Observation~\ref{columnsCharacterization}. So eventually, we show by induction on $d$ that the assumption in our theorem implies that $V$ can be partitioned into $V_{1},V_{2},\ldots,V_{{2^d}}$ such that every induced subgraph $G \left[ V_{i} \right]$ is isomorphic to a linear forest. As the basis of the induction we use the proof of a weaker theorem proven in~{\cite{cac11}}.

For $d=1$, the graph $G$ is either a complete graph on four vertices or, according to Brooks' theorem, $G$ can be colored with three colors. We know that the graph $K_{4}$ can be drawn on two columns, so the statement holds  in the first case. In the second case, the vertices of $G$ can be partitioned into three color classes $C_{1}$, $C_{2}$ and $C_{3}$ (we label the colors as $c_{1}$, $c_{2}$ and $c_{3}$). Consider the induced subgraph $G\left[C_{1}\cup C_{2}\right]$. If there is a vertex of degree three, then we color it with the color $c_{3}$. Thus we ensured that $\Delta \left( G \left[ C_{1} \cup C_{2} \right] \right) \leq 2$. If there is a cycle left, then we choose its arbitrary vertex and color it with the new color $c_{4}$. Afterwards, the graph $G\left[C_{1}\cup C_{2}\right]$ is isomorphic to a linear forest, but there might be a vertex of degree three in the graph $G\left[C_{3}\cup C_{4}\right]$. If there is such vertex, then we color it to $c_{1}$. After that, the graphs $G\left[C_{1}\cup C_{2}\right]$ and $G\left[C_{3}\cup C_{4}\right]$ are both linear forests.

Now we do the inductive step. Let the maximal degree of $G$ is at most $2^{d+1}-1$. Then, according to Lemma~\ref{lovaszLemma}, $V$ can be partitioned into $V_{1}$ and $V_{2}$ such that $G\left[ V_{1} \right] \leq 2^{d}-1$ and $G \left[ V_{2} \right ] \leq 2^{d}-1$, if we set $m=2$ and $k_{1}=k_{2}=2^{d}-1$. It follows from the inductive step that the vertices of each of the graphs $G \left[ V_{1} \right]$, $G \left[ V_{2} \right]$ can be partitioned into $2^{d-1}$ required sets. Together these partitions give the partition of $V$ into $2^{d-1}+2^{d-1}=2^{d}$ sets.
\end{proof}

Note that the reverse implication does not hold, as every star graph can be located on two columns in the plane and its maximal degree does not have to be bounded.

\section{Mixed Colorings}

We saw that drawing/locating of a graph with bounded number of columns is related to a special form of defective coloring where every color class induces either an independent set or a linear forest. Such coloring is called {\sl mixed} and we use it later to prove $\NP$-completeness of a problem of deciding whether a graph can be drawn/located on $l \geq 2$ columns.

Coloring of $G$ with only path colors is called {\sl path coloring} and, on the other hand, coloring with only normal colors is called {\sl normal coloring}. If we can color a graph $G$ with $a$ normal colors and $b$ path colors, then we say that $G$ is $\left( a,b \right)$-colorable. The class of all $\left( a,b \right)$-colorable graphs is denoted as $\ensuremath{\mathcal{G}}_{a,b}$ and it is referred as a {\sl mixed coloring type}. 

Then we see that $\ensuremath{\mathcal{G}}_{a,b}\supseteq\ensuremath{\mathcal{G}}_{c,d}$ if and only if there is a sequence $\left\{ \ensuremath{\mathcal{G}}_{ai,bi}\right\} _{i=1}^{n}$ such that $a_{1}=a$, $b_{1}=b$, $a_{n}=c$ , $b_{n}=d$ and for every $i\in\left\{1,\ldots,n-1\right\}$ it holds that $a_{i+1}=a_{i}+2$, $b_{i+1}=b_{i}-1$ or $a_{i+1}=a_{i}-1$, $b_{i+1}=b_{i}+1$. That is, there is a sequence of steps where every step corresponds to a substitution of one path color by two normal colors or one normal by one path color.

We can consider the partially ordered set of the set of all mixed coloring types ordered by inclusion. The picture bellow shows the modified Hasse diagram of this POSET where the inclusion corresponds to an oriented path between two types. The inclusion is not total order in this case as there are incomparable elements.

\begin{figure}[h!]
	\centering
	\includegraphics{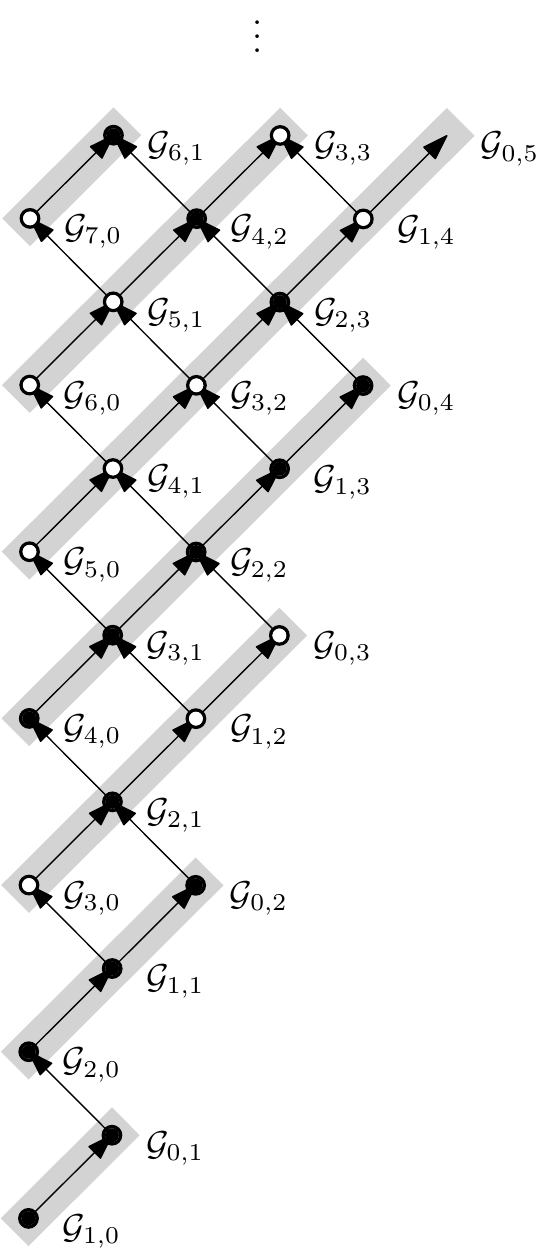}
	\caption{Mixed coloring types ordered by inclusion}
	\label{fig:POSET}
\end{figure}

According to Observation~\ref{columnsCharacterization}, the mixed coloring types which are drawn in the common grey site are classes of graphs that can be drawn on the same number of columns. Similarly, the mixed coloring types denoted as black vertices correspond to the graph classes from Theorem~\ref{locatingCharacterization}.

The Four Color Theorem implies that every planar graph is $\left( 4,0 \right)$-colorable and Wayne Goddard~{\cite{god91}} showed that it is also $\left( 0,3 \right)$-colorable. Thus we get the following corollary.

\begin{cor}
\label{planar3Columns}
Every planar graph can be drawn on three columns.
\end{cor}

C\'{a}ceres et.~al.~{\cite{cac11}} showed that every outerplanar graph can be drawn (and located) on two columns. In the same paper there is introduced an example of a planar graph which is not $\left( 2,1 \right)$-colorable. Thus we need four columns to locate an arbitrary planar graph. The natural question is whether every planar graph is $\left( 1,2 \right)$-colorable. The following proposition shows that using one normal and two path colors is insufficient too.

\begin{prop}
\label{planarNot(1,2)}
There is a planar graph which is not $\left( 1,2 \right)$-colorable.
\end{prop}
\begin{proof}
Let $\alpha$ be the normal color and $\beta$ and $\gamma$ be the path colors we can use. Consider the gadget $H$ depicted in part a) of Figure~\ref{fig:constructionPlanarNot(1,2)}. This gadget is isomorphic to a complete graph on four vertices with a path on ten vertices inside each inner face. The path colors $\beta$ and $\gamma$ cannot both appear on the vertices of the outer face otherwise it is not possible to color the path adjacent to them. We could color at most four vertices of this path with $\beta$ and $\gamma$ in such case, but there would still be an edge with both vertices of color $\alpha$. But this is not possible, since $\alpha$ is normal color. Thus the vertices of the outer face are colored with $\alpha$ and one path color, say $\beta$.

\begin{figure}[h!]
	\centering
	\includegraphics{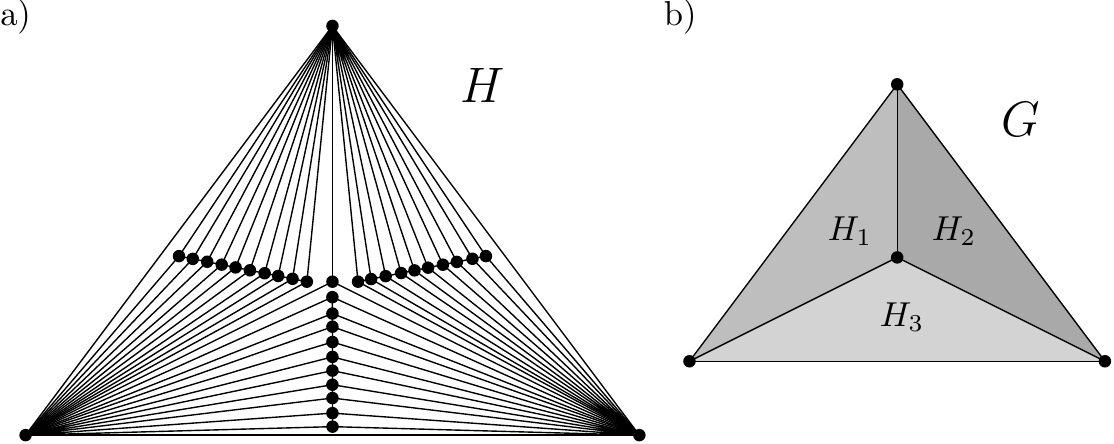}
	\caption{Construction of a planar graph which is not $\left( 1,2 \right)$-colorable}
	\label{fig:constructionPlanarNot(1,2)}
\end{figure}

Now we join three copies $H_{1}$, $H_{2}$ and $H_{3}$ of $H$ as shown in Figure~\ref{fig:constructionPlanarNot(1,2)}, part b), and we obtain the graph $G$. We see that $G$ is not $\left( 2,1 \right)$-colorable, because the only way how to color it with $\alpha$, $\beta$ and $\gamma$ is to color $K_{4}$ with $\alpha$ and $\beta$ and this is clearly not possible.
\end{proof}

It is not difficult to prove that there is an outerplanar graph which is not $\left( 1,1 \right)$-colorable, hence we know the tight estimations on mixed colorability of both planar and outerplanar graphs.

Now our main goal is to prove $\NP$-completeness of problem of deciding whether a graph $G$ is $\left( a,b \right)$-colorable for sufficiently large $a$ and $b$. As a consecutive result we obtain that drawing/locating of graphs on bounded number of columns is a difficult task answering the open question in~{\cite{cac11}}. 

This problem is already partially solved as Glenn G. Chappell, John Gimbel and Chris Hartman~{\cite{chap06}} proved that determining whether $G$ can be colored with $l \geq 2$ path colors is $\NP$-complete. Although this does not answer the question for locating of graphs (we need to prove the statement for general mixed colorings, not only for path colorings), we later apply a similar technique to prove $\NP$-completeness of $\left( a,b \right)$-colorability for sufficiently large $a$ and $b$. 

In the following lemma we prove the initial case by using a reduction to the One-in-three 3SAT problem (see~{\cite{gar90}}).

\begin{lem}
\label{(1,1)NP}
It is $\NP$-complete to decide whether or not a graph $G= \left( V,E \right)$ is $\left( 1,1 \right)$-colorable.
\end{lem}
\begin{proof}
Let $F$ be a collection of $m$ clauses $C_{1}, C_{2},\ldots,C_{m}$ over $n$ Boolean variables $v_{1},v_{2},\ldots,v_{n}$ such that each clause $C_{i}$ contains exactly three literals $c_{i,1}$, $c_{i,2}$ and $c_{i,3}$. Each literal $c_{i,j}$, $i \in \left[m\right]$ and $j \in \left\{1,2,3\right\}$, is either $v_{k}$ or $\overline{v_{k}}$ for some suitable $k \in \left[n\right]$. One-in-three 3SAT is a problem of determining whether there is a truth assignment $e$ satisfying $F$ such that each clause in $F$ has exactly one true literal (and thus exactly two false literals).

We construct a graph $G \left( k \right)$ shown in Figure~\ref{fig:G(k)} for each variable $v_{k}$. Then, for each clause $C_{i}$, we construct a graph $G \left( C_{i} \right)$ which is isomorphic to $K_{3}$ and each one of its vertices represents a different literal of the clause $C_{i}$.  Let $G \left( F \right)$ be a graph consisting of all the graphs $G \left( k \right)$ and $G \left( C_{i} \right)$ where the vertex $c_{i,j}$ is adjacent to $v \in V \left( G \left( k \right) \right)$ if and only if the literal $c_{i,j}$ is $v \in \left\{ v_{k}, \overline{v_{k}} \right\}$.
\begin{figure}[h!]
	\label{fig:G(k)}
	\centering
	\includegraphics{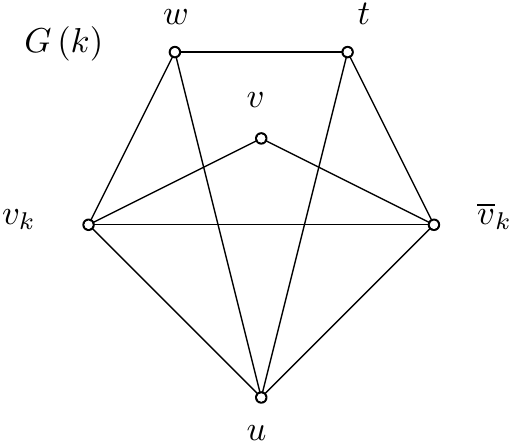}
	\caption{The graph $G \left( k \right)$}
\end{figure}

Suppose that $G$ is colored with one path and one normal color, say black and white. Then the vertices $v_{k}$ and $\overline{v}_{k}$ of $G \left( k \right)$ are colored differently. Otherwise they are black and the vertex $u$ must be white. But then, since white is a normal color, $w$ and $t$ are black and induce a black 4-cycle together with $v_{k}$ and $\overline{v}_{k}$. Also, if the vertices $x \in \left\{ v_{k}, \overline{v}_{k} \right\}$ and $c_{i,j}$ are adjacent, then their colors are different too. Assume to the contrary that $x$ (say $x=v_{k}$) and $c_{i,j}$ are both black and adjacent. Then we know that $\overline{v}_{k}$ is white and thus $u$ and $v$ are black. Hence $v_{k}$ has three black neighbors which is a contradiction.

We define the truth assignment $e$ for $F$ as follows: if $v_{k}$ is white, then $e \left( v_{k} \right)$ is true else $e \left( v_{k} \right)$ is false. The assignment $e$ is correct as the vertices $v_{k}$ and $\overline{v}_{k}$ are not monochromatic. In addition, there is exactly one true literal in every clause. Otherwise there would be a black 3-cycle or an edge with both vertices white in some $G \left( C_{i} \right)$.

Suppose that $e$ satisfies $F$ such that every clause has exactly one true and two false literals. Then we color the labeled vertices of each $G \left( C_{i} \right)$ white, if the corresponding literal is true; otherwise black. By the assumption, there is no monochromatic graph $G \left( C_{i} \right)$. After that, we color the vertex $v \in \left\{ v_{k}, \overline{v}_{k} \right\}$ adjacent to $c_{i,j}$ black (white, respectively) if $c_{i,j}$ is white (black, respectively). Note that the vertices $v_{k}$ and $\overline{v}_{k}$ are, again, differently colored. It remains to color the rest of graph $G \left( k \right)$ for each $k \in \left[n\right]$.
\end{proof}

We use a reduction to the Colorability Problem in the final statement, but this problem is $\NP$-complete for at least three colors, thus we need to consider one more special case. That is $\left( 0,2 \right)$-colorability. Although the following lemma is already known to be true~{\cite{chap06}}, the known proof is based on the result with so called one-defective colorings. For completeness we include a short proof which uses a similar idea as the previous one (a variation of a technique used by Ho\`{o}ng-Oanh Le {\cite{le03}}).

\begin{lem}
\label{(0,2)NP}
It is $\NP$-complete to decide whether or not a graph $G= \left( V,E \right)$ is $\left( 0,2 \right)$-colorable.
\end{lem}
\begin{proof}
The main idea is the same as before. We use a reduction to a variation of 3SAT problem, only this time we use Not-All-Equal 3SAT (see {\cite{gar90}}). It is a problem of determining whether there is a truth assignment satisfying a formula such that each clause has at least one true literal. So, let the notation be the same as in Lemma~\ref{(1,1)NP} with the only difference that instead of $G \left( k \right)$ we use the graph depicted in Figure~\ref{fig:newG(k)}.

\begin{figure}[h!]
	\label{fig:newG(k)}
	\centering
	\includegraphics{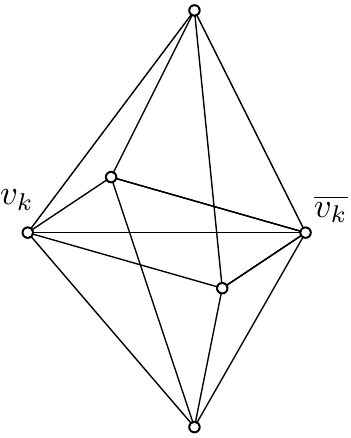}
	\caption{The new graph $G \left( k \right)$}
\end{figure}

Let $G$ be colored with two path colors black and white. One can easily show that it holds again that the vertices $v_{k}$ and $\overline{v}_{k}$ have distinct colors. Otherwise the remaining vertices of $G \left( k \right)$ induce a monochromatic 4-cycle. The adjacent vertices $x \in \left\{ v_{k}, \overline{v_{k}} \right\}$ and $c_{i,j}$ are also heterochromatic. Otherwise $x$ would have three neighbors of the same color.

Now, we can define the truth assignment as follows: if $v_{k}$ is white, then $e \left( v_{k} \right)$ is true else $e \left( v_{k} \right)$ is false. The previous facts imply correctness of this assignment and there is at least one true literal in every clause, otherwise $G\left( C_i \right)$ would be monochromatic 3-cycle. The proof of the reverse implication is analogous too.
\end{proof}

\begin{thm}
\label{mixedColoringNP}
It is $\NP$-complete to decide whether or not a graph $G= \left( V,E \right)$ is $\left( a,b \right)$-colorable where $a+b \geq 2$ and $\left(a,b\right) \neq \left(2,0\right)$.
\end{thm}
\begin{proof}
We apply  a reduction to the Colorability Problem. That is, a problem of determining whether or not it is possible to color a given graph $G$ with $k$ colors. If we set $k=a+b$, then we can assume, according to the previous lemmas, that $k \geq 3$. The Colorability Problem is $\NP$-complete in such case, thus we can consider the reduction. Suppose that $G$ is a given graph. Let us create the graph $H$ by joining two disjoint copies of the complete graph $K_{a+2b-1}$ to every vertex $v$ of $G$.

Suppose that $G$ is colored with $k$ normal colors. Then we color the cliques for every vertex $v$ with all colors. Two vertices per path color and one per normal color.

On the other hand, if $H$ is colored with $a$ normal and $b$ path colors, then $G$ is colored with at most $k=a+b$ normal colors. Assume to the contrary that there is an edge $uv$ with both vertices colored with the same path color (say black) in $G$. Then $u$ has at least three black neighbors, because the sizes of the adjacent cliques imply that there is at least one other black vertex in every one of them. This is a contradiction since the coloring of $H$ is correct.
\end{proof}

\begin{cor}
\label{drawingNP}
It is $\NP$-complete to decide whether or not it is possible to draw a given graph on $l \geq 2$ columns.
\end{cor}

\begin{cor}
\label{locatingNP}
It is $\NP$-complete to decide whether or not it is possible to locate a given graph on $l \geq 2$ columns (in a grid of sufficiently large dimension).
\end{cor}

\section{Planar Grid Drawings}

Although Theorem~\ref{complexityTheorem} and the Four Color Theorem imply that every planar graph is locatable, the drawings obtained by this approach do not have to be planar. On the other hand, De Fraysseix, Pach, and Pollack~\cite{fray90}, Schnyder~\cite{schnyder90}, and Chrobak and Nakano~\cite{chrobak95} proved that any planar graph on $n$ vertices has a planar grid drawing which can be realized in grids of sizes $\left( 2n - 4\right) \times \left( n - 2\right)$, $\left( n - 2\right) \times \left( n - 2\right)$ and $\left\lfloor 2\left( n - 1\right)/3 \right\rfloor \times \left( 4\left\lfloor 2\left ( n - 1\right)/3 \right\rfloor-1\right)$, respectively. Unfortunately, these drawings are not primitive.

\begin{defn}
A primitive planar grid drawing is said to be {\sl proper}.
\end{defn}

In this section we show that the Four Color Theorem together with F\'{a}ry's theorem imply the existence of a proper grid drawing for every planar graph.

\begin{thm}
\label{proper}
There exists a proper grid drawing for every planar graph.
\end{thm}

\begin{proof}
The main idea is to map a planar drawing of a graph, where line segments correspond to edges, to a grid such that no line segment contains more than two grid points. To find convenient coordinates we use the Four Color Theorem.

Let $G=\left( V,E\right)$ be a planar graph and let $\phi \left( G \right)$ be its initial planar embedding whose existence is ensured by, for example, F\'{a}ry's theorem.  The mapping $\phi$ maps vertices of $G$ to points with real coordinates in the plane. The edge $uv \in E$ corresponds to the line segment $\overline{\phi \left( u \right)\phi \left( v \right)}$ in the embedding $\phi \left( G \right)$. Let $f \colon V \to C$ be a vertex coloring of $G$ with four colors and let $C=\left\{ \left( 0,0 \right), \left( 0,1 \right), \left( 1,0 \right), \left( 1,1 \right) \right\}$. The first coordinate of color $c \in C$ is denoted as $c_{1}$, the second one as $c_{2}$. The existence of $f$ is ensured by the Four Color Theorem.

Let $r \in \mathbb{R}$ denote the smallest distance such that every vertex can be shifted by $r$ in any direction so that the condition on planarity holds still. We can set $r$ as one-half of the minimum distance between two points $x$, $y \in \mathbb{R}^{2}$ such that $x$ and $y$ belong to line segments which represent two vertex disjoint edges of $G$. The distance $r$ is positive, otherwise we get a contradiction with planarity of $\phi \left( G \right)$. Thus, for every vertex $v \in V$, there is an open neighborhood $\Omega \left( v,r \right)$ of the point $\phi \left( v \right)$ such that any point $x \in \Omega \left( v,r \right)$ can represent the vertex $v$ without violating the condition on planarity. Let us assume that no vertical line segment intersects two different neighborhoods $\Omega \left( u,r \right)$, $\Omega \left( v,r \right)$. Otherwise we can lower the distance $r$ as no two points $\phi \left( u \right)$, $\phi \left( v \right)$ lie on the same vertical line.

Now we put vertical lines across the whole plane such that the distance between two consecutive lines is $\epsilon>0$. We choose the number $\epsilon$ such that every neighborhood is crossed by at least six lines (we can assume that $\epsilon=1$). Then we choose one line and declare it as the initial line. Each line gets number according to its order, the initial line has number zero. Now for every vertex $v \in V$, we set $\phi \left( v \right)=x$, where $x$ is a point from $\Omega \left( v,r \right)$ such that it lies on some vertical line with number $l$ and $l \equiv f \left( v \right)_{1} \left( \bmod 2 \right)$, $l \equiv f \left( v \right)_{1} \left( \bmod 3 \right)$. We can always choose such line, because there are six consecutive lines crossing the neighborhood $\Omega \left( v,r \right)$. Thus numbers of these lines get through all values modulo two and three. In the rest of the proof, we assume that the first coordinates of points representing the vertices of $G$ are integers. The point $x$ is in $\Omega \left( v,r \right)$, so the modified embedding is still planar. By choosing appropriate lines we can also ensure that no two adjacent vertices lie on the same vertical line (but we might have to cross the neighborhoods by twelve lines).

\begin{figure}[h!]
	\centering
	\includegraphics{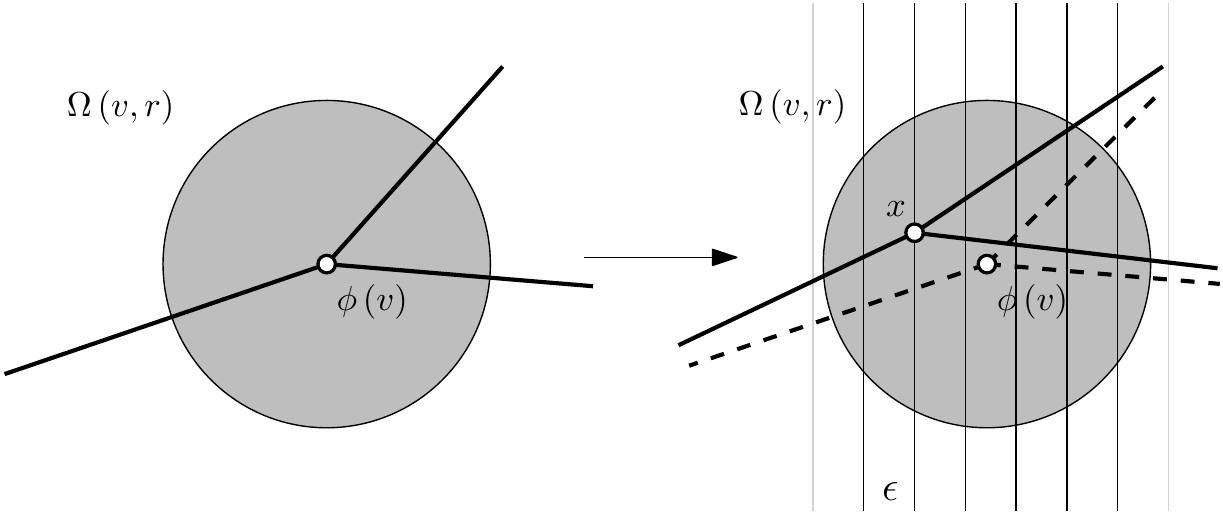}
	\caption{Placing the vertical lines}
	\label{fig:proper1}
\end{figure}

Let $P$ denote the set of all prime numbers which appear in the decomposition of the difference $\left| \phi \left( u \right)_{1} - \phi \left( v \right)_{1} \right|$ where $\phi \left( u \right)_{1}$, $\phi \left( v \right)_{1}$ are the first coordinates of points $\phi \left( u \right)$, $\phi \left( v \right)$ and $uv \in E$. The set $P$ is finite, because no two points representing vertices lie on the same vertical line and thus the difference is always positive. Now we analogously put horizontal lines across the whole plane such that the distance between two consecutive lines is $\delta>0$. This time we choose $\delta$ such that every vertical line is crossed by at least $\prod_{p \in P}{p}$ lines in every neighborhood. 

\begin{figure}[h!]
	\centering
	\includegraphics{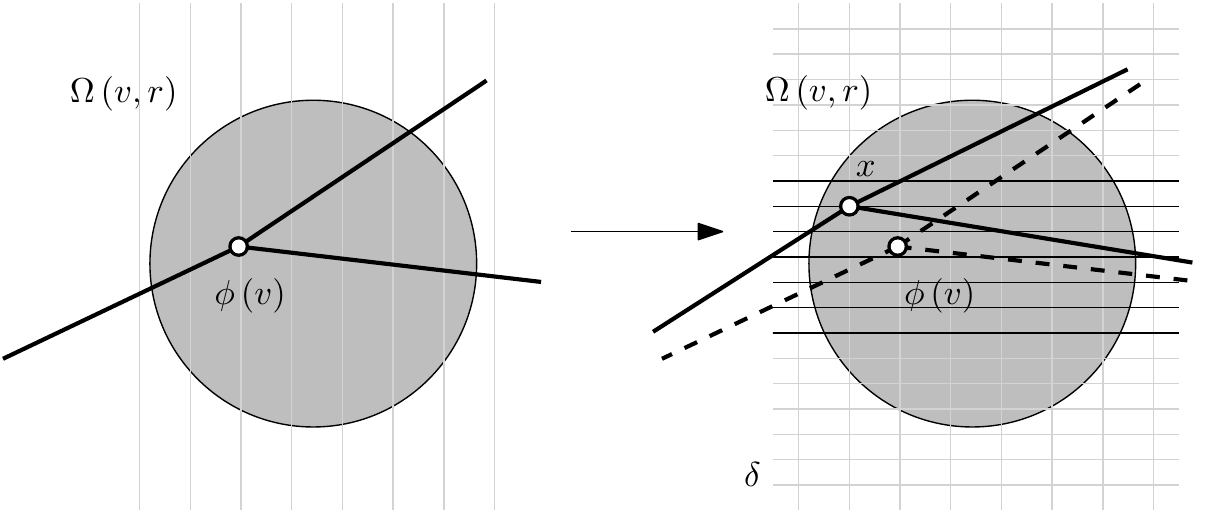}
	\caption{Placing the horizontal lines}
	\label{fig:proper2}
\end{figure}

Again, we declare one of these lines as initial and number them according to their order. Then, for every vertex $v \in V$, we set $\phi \left( v \right)=x$ such that $x \in \Omega \left( v,r \right)$, the first coordinate of $\phi \left( v \right)$ remains the same and $x$ lies on the horizontal line with number $l$, where $l \equiv f \left( v \right)_{2} \left( \bmod 2 \right)$, $l \equiv f \left( v \right)_{2} \left( \bmod 3 \right)$. In addition, if there is another prime number $p$ which divides the difference  $\left| \phi \left( u \right)_{1}-\phi \left( v \right)_{1} \right|$, $uv \in E$, then we set such horizontal lines for $u$ and $v$ that their numbers are not congruent modulo $p$. The different residues modulo $p$ can be chosen according to the coloring $f$. Each color corresponds to an unique residue modulo $p$ ($p>4$, so there is enough residues). We chose $\delta$ such that there is enough horizontal lines from which we can always choose the right ones.

Eventually the horizontal and vertical lines form an elongated grid which we can modify into a regular grid. It suffices to contract the grid such that the size of columns equals the size of rows, that is $\epsilon=\delta$. The contraction does not violate planarity, because the whole grid is regularly contracted, thus no positive distance can lower to zero. The coordinates of points are chosen such that every line segment is primitive, thus the embedding is planar and primitive.
\end{proof}

This result gives an affirmative answer to the conjecture asked by Pe\H{n}aloza and Martinez~\cite{pen09}. The authors point out that proof of this statement would yield an alternate proof of the Four Color Theorem. However we use it as one of the assumptions. In fact, this theorem is equivalent to the Four Color Theorem, as the proof of the reverse implication is apparent. Also if we use the Five Color Theorem in the proof instead, then we obtain three-locatable planar grid drawings of planar graphs and thus the problem of finding almost-proper grid drawings (i.e.\ at most three grid points on each line segment) belongs to $\P$. 

Note that the choice of coordinates also gives us a coloring of $G$ with at most four colors. In fact, we also proved a stronger conjecture from~\cite{pen09}.
\begin{cor}
\label{properCor}
Any planar graph $G$ is isomorphic to a plane subgraph $H$ of the visibility graph of the integer lattice, in such a way that the function $g\left( a_{1},a_{2} \right)=\left( a_{1} \left( \bmod 2 \right), a_{2} \left( \bmod 2 \right) \right)$ is a coloring of $H$ that uses exactly $\chi\left( G \right)$ colors.
\end{cor}

The grid drawings obtained by the proof can require large area with no reasonable bounds. However if we start with a nicer intial drawing, then we can estimate the upper bounds quite easily. 

Suppose that the initial embedding is already a grid drawing of size $O \left( n \right) \times O \left( n \right)$ where $n$ denotes the number of vertices of a given graph. The results of Chrobak, De Fraysseix, Pach, and Pollack, and Nakano~\cite{ chrobak95, fray90, schnyder90} ensure the existence of such embedding. Then the following lemma gives us a lower bound on $r$.

\begin{lem}
\label{distance}
Given an $n \times n$ integer grid, $n > 1$, the minimum nonzero distance from any grid point to any line segment is in $\Omega \left( 1 \over n \right)$. 
\end{lem}
\begin{proof}
Let us recall that the distance from a point $\left( a_{1},a_{2} \right)$ to a line with equation $kx + ly + m = 0$ is given by the formula
\[ {\left| ka_{1} + la_{2} + m \right|} \over{\sqrt{k^2 + l^2}}.\]
Without loss of generality let us assume that the first point is $\left( a_{1},a_{2}\right)$ and the line intersects grid points $\left( 0,0 \right)$ and $\left( b_{1}, b_{2} \right)$ where $b_{1}$ and $b_{2}$ are relatively prime (otherwise we consider the point $\left( \frac {b_{1}}{\gcd \left( b_{1},b_{2} \right)}, \frac {b_{2}}{\gcd \left( b_{1},b_{2} \right)}  \right)$ that lies on the same line). Then the equation of our line is $b_{1}x - b_{2}y = 0$ and $\left| b_{2}a_{1} - b_{1}a_{2} \right|$ is at least one. Therefore the minimum nonzero distance is at least
\[ \frac{1}{\sqrt{b_{1}^2 + b_{2}^2}}.\]
Now we minimalize the expression by choosing coordinates $b_{1}$ and $b_{2}$. The sum $b_{1}^2 + b_{2}^2$ is maximal when $b_{1}$ and $b_{2}$ differ as little as possible, so the appropriate choice is $b_{1} = n$ and $b_{2} = n-1$. So the minimum possible distance from grid point to a line is at least
\[ \frac{1}{\sqrt{2n^2 - 2n + 1}} \in \Omega \left( {1} \over {n} \right).\]
\end{proof}

Thus if the size of the initial grid drawing is $cn \times cn$, where $c > 0$ is some constant, then the minimum nonzero distance $r$ from any point representing a vertex to any point representing an edge is in $\Omega \left( 1 \over n \right)$. In the first part of the proof we refine the coordinates such that the neighborhood of every vertex is intersected by a constant number of vertical lines. The diameter of the neighborhoods is exactly $r \in \Omega \left( 1 \over n \right)$, therefore the width of the new grid drawing is in $O \left( n^2 \right)$.

All that is left is to estimate the height of the drawing. Following the proof we refine the vertical coordinates such that every neighborhood is intersected by at least $\prod_{p \in P}{p}$ horizontal lines. The diameter of the neighborhoods  is now in $O \left( 1 \right)$, so if we find a function $f \colon \mathbb{N} \to \mathbb{N}$ such that the product $\prod_{p \in P}{p}$ is in $O\left( f \left( n \right) \right)$, then we know that the height is in $O\left( n^2 f \left( n \right) \right)$ too.

We can focus on every vertex separately. Let $v$ be a vertex of $G$ and let $P_{v}$ denote the set of prime numbers which divide the nonzero horizontal distance between the points $\phi \left(u \right)$ and $\phi \left(v \right)$ where $uv \in E$. Then the product of primes which divide the distance between $u$ and $v$ is in $O \left( n^2 \right)$ as it is the width of the whole drawing. Therefore we get that \[ \prod_{p \in P_{v}}{p} \in O \left( n^{2 d \left( u \right) } \right)  \] where $d \left( u \right)$ denotes the degree of $u$. According to the Chinese Remainder Theorem, we see that we can consider only the vertex with maximum degree $\Delta$.

Hence we can find a proper grid drawing of any planar graph $G$ with given coloring in the grid of size $O\left( n^2\right) \times O\left( n^{2 \Delta + 2}\right)$ where $n$ denotes the number of vertices of $G$. Thus the rough estimation of the size of the drawing is polynomial for $\Delta \in O \left( 1 \right)$, quasi-polynomial for $\Delta \in O \left( \poly \left( \log n \right) \right)$ and exponential for linear maximum degree.

Unfortunately we don't know how to embed the general planar graphs in a grid of polynomial size and the following question remains open.

\begin{conj}
\label{properConjecture}
For arbitrary planar graph $G$, is there a proper grid drawing of $G$ in a grid of polynomial size?
\end{conj}

\section{Conclusion}

We studied grid drawings from three points of views. First, we showed a connection between the chromatic number of the graph $G$ and the maximal number of grid points that must appear on a line segment of a grid drawing of $G$. This led to a new classification of graphs according to so called locatability. 

Second, we showed that it is $\NP$-complete to find the minimal number of columns on which a graph can be drawn. If we consider only primitive grid drawings, then we have to move to higher dimensions as the chromatic number grows. We also characterized the graphs which can be located on $l$ columns in $d$-dimensional grid and showed that locating  graphs is also $\NP$-complete. Natural question is what happens if we consider grid drawings with both width and height bounded~\cite{wood04}. Such problem is closely connected to "No-three-in-line problem"~\cite{guy68}.

In the last section we proved that there exist primitive planar grid drawings of an arbitrary planar graph. However the proof of this statement uses a strong result, namely the Four Color Theorem. Perhaps the most intriguing question left open is whether there is a proof of this statement without using the Four Color Theorem. Such proof would yield an alternate proof of this classical result in the graph theory.

\section*{Acknowledgments}

{\small
I would like to thank my supervisor Pavel Valtr for his time and for all the provided advice.}



\small 
\bibliographystyle{abbrv}


\bibliography{citace}


\end{document}